\documentclass[twoside]{amsart}
\usepackage{latexsym}
\usepackage{amssymb,amsmath,amsopn,setspace}
\usepackage[dvips]{graphicx}   
\usepackage{color,epsfig}      

\newcommand{\mathsym}[1]{{}}
\newcommand{\unicode}[1]{{}}

\newtheorem{theorem}{Theorem}[]

\theoremstyle{definition}
\newtheorem{defi}[]{Definition}[]
\newtheorem{remark}[]{Remark}[]

\numberwithin{equation}{section}

\begin{document}

\title[On the non-Paschian ...]{On the non-Paschian ordered planes}
\author[\'A. G.Horv\'ath]{\'Akos G.Horv\'ath}
\date{Febr. 2016}

\address{\'A. G.Horv\'ath, Dept. of Geometry, Budapest University of Technology,
Egry J\'ozsef u. 1., Budapest, Hungary, 1111}
\email{ghorvath@math.bme.hu}


\begin{abstract}
We give a non-Paschian plane based on the property of betweenness which cannot be derived from an ordering of the points of a line. In this model there is no possibility to define the congruence of segments but we can define angle, triangle and angle measure, respectively. With respect to our definitions the plane has an elliptic character, meaning that the sum of the angles of a triangle is greater than $\pi$.
\end{abstract}

\maketitle

\section*{Introduction}

On the Cartesian plane there are two methods to define an ordered plane on which the Pasch axiom is false. If we leave out a point from an Euclidean plane we can give such a plane as a simple example. Using this method we automatically omit some further axioms from the axiom system of the Euclidean plane. A line through the missing point cannot fulfil the first and third axioms of congruency (namely, that every segment can be "laid off" on a given side of a given point on a given line in at least one way, and the axiom of the addition of congruent segments), respectively. In addition, the Cantor axiom (and consequently the Dedekind axiom) is also false on this plane.

Another method has been shown by Szczerba (see \cite{szczerba} and \cite{adler}) who defined the property of "betweenness" based on an exotic ordering of the points on a line. Applying this method on the Cartesian plane we can define a plane in which -- except the Pasch axiom -- all axioms of the Euclidean plane hold.  We describe this method briefly. Let $f:\mathbb{R}\rightarrow \mathbb{R}$ be an additive onto function and $0<f(1)$. Let $x<^\star y$ if $f(x)<f(y)$. Then $(\mathbb{R},+,\cdot,<^\star)$ is a semi-ordered field, means that it is an ordered additive group but it is not hold necessarily that if $0<^\star x$ and $0<^\star y$ then $0<^\star xy$. Szczerba proved that the Cartesian plane over $(\mathbb{R},+,\cdot,<^\star)$ satisfies the axioms of plane geometry (with the full second-order continuity axiom) except the Pasch axiom. Moreover if the ordering is not the usual one (based on an $\mathbb{R}$-linear mapping $f$) then the corresponding Cartesian plane  does not satisfy the Pasch axiom. To define such a non-usual ordering we have to define an additive non-linear function. Szczerba's method based on the Hammel basis of the reals over the rationals and hence it uses the axiom of choice. Addler certified in \cite{adler} that the axiom of choice plays an essential role in Szczerba's proof, namely all models satisfying the axioms of the Euclidean plane except the Pasch axiom are isomorphic to the model constructed by Szczerba. Addler noted also that the construction of additive non-linear functions implies the existence of non-measurable sets in the sense of Lebesgue (such a function is not Lebesgue measurable). Hence if we change the axiom of choice to the axiom of determinacy then all additive functions are linear ones, which means that all models satisfying the remaining axioms of the plane geometry fulfil the Pasch axiom too.

Our purpose is to define a non-Paschian plane based on such well-defined property of betweenness which cannot be derived from an ordering of the points lying on a line.  In the new model it is not possible to define the congruence of segments but we can define angle, triangle and angle measure, respectively. With respect to our definitions the plane has an elliptic character, meaning that the sum of the angles of a triangle is greater than $\pi$. It is interesting that the continuity axioms hold as well. Finally, contrary to the plane of Szczerba our construction is independent from the axiom of choice, namely it remains a non-Paschian plane even if we substitute the axiom of choice with the axiom of determinateness.

There are several books on the foundations of geometry. We mention here some of them as references for the interested reader. We propose the books \cite{bonola}, \cite{coxeter}, \cite{hilbert}, \cite{berger}, \cite{gho} to investigate non-Euclidean geometries in general and the papers  \cite{gho 1}, \cite{gho 2}, \cite{gho 3} for a deeper investigation of the hyperbolic plane.

\section*{A "betweenness"-based ordered plane}

In this paper the concept of \emph{plane} has a double meaning, on one hand, it is a so-called \emph{set of points} and on the other hand, it is a set of \emph{lines}, thus we denote by $\mathcal{P}$ and $\mathcal{L}$, respectively. We would like to allow on the plane the binary notion of \emph{incidence} to points and lines; and the ternary notion of \emph{betweenness} defined on the set of points which are lying on the same line. We use Hilbert's axiom system (see \cite{hilbert}) which is the most popular and used one in the fundamentals of geometry.

Let us denote the binary relation of {\em incidence} of the point $A$ and the line $a$ by $AIa$ where $A\in \mathcal{P}$ and $a\in \mathcal{L}$. The axioms of incidence are formulated as follows:

\begin{enumerate}
 \item[I1.] For any two points $A$ and $B$ there exists a line $a$ which fulfil $AIa$ and $BIa$.
 \item[I2.] For any two distinct points there exists only one line that contains both; consequently, if $AIa$ and $BIa$ and also if $AIb$ and $BIb$ are fulfilled then $a\equiv b$. (Hence the line $a$ determined uniquely by $A$ and $B$, and it can be denoted by $AB$.)
 \item[I3.] There exist at least two points on a line; for all $a\in \mathcal{L}$ there are $A,B\in \mathcal{P}$ such that $AIa$ and $BIb$. There exist at least three points which do not lie on the same line.
 \end{enumerate}

For three points $A,B,C$ the relation that the point \emph{$B$ is between the points $A$ and $C$} we denote by $(ABC)$.
The axioms of betweenness are
\begin{enumerate}
\item[B1.] If $(ABC)$, then $(CBA)$, and there is a line $e$ which contains the points $A,B,C$.
 \item[B2.] If $A$ and $C$ are two distinct points, there is at least one point $B\in AC$ for which $(ACB)$.
 \item[B3.] Of any three points on a line there exists only one which lies between the other two.
 \item[B4.](Pasch's Axiom): Let $A$, $B$, $C$ be three distinct points not collinear and let $a$ be a line not passing through any of the points $A$, $B$, $C$. If there exists a point $D$ lying on the line $a$ for which $(ADB)$ then either there is a point $E$ such that $EIa$ and $(BEC)$ or there exists a point $F$ such that $FIa$ and $(AFC)$.
\end{enumerate}

We note that Hilbert's axiom system has a weaker form, where the existence part of $\mathrm{B}2$ ("there is at least one") is proven by a theorem. As we can see from these axioms (using Tarski's apellations of \cite{tarski}) we have a Lower 2-dimensional axiom ($\mathrm{I}3$), but we do not use Upper 2-dimensional axiom. In fact, the Pasch's Axiom ($\mathrm{B}4$) in the present form restrict the value of the usual concept of dimension to $n=2$ and so a natural definition arise for ordered plane of dimension $2$ if we assume the above seven axioms. If we would like to consider the dimension open then we have to modify the strong axiom $\mathrm{B}4$ on such a way that it will be effect only on "2-dimensional configurations". This motivates the following modification.

\begin{enumerate}
\item[$\mathrm{B}4^\star$.]: Let $A$, $B$, $C$ be three distinct points not collinear and let $a$ be a line not passing through any of the points $A$, $B$, $C$. If there exists a point $D$ lying on the line $a$ for which $(ADB)$ and $a$ intersects the remaining two lines $(BC)$ and $(AC)$ in the points $E$ and $F$, respectively then either $(BEC)$ or $(AFC)$.
\end{enumerate}

Clearly, $\mathrm{B}4^\star$ is a weakening of the requirement of $\mathrm{B}4$. For example, we do not investigate its validity on lines which intersects only two lines from the fixed three ones. Hence we exclude those counterexamples on Pasch's axiom in what a possible point of intersection has been dropped from the plane. We also disregard from those lines which in a case of the extension of the axiom system to the axiom system of a higher dimensional space cannot be considered as a line of the plane of the points $A,B,C$.

On the base of this weakening it is possible that we cannot fit together the betweenness property of three lines, because there is no a fourth one which intersects all of them in distinct points. (The original strong form of Pasch's axiom excludes this possibility.) Hence we introduce a fourth axiom of incidence which is in the case when the original Pasch's axiom holds is a consequence of the axioms.

\begin{enumerate}
\item[I4.]: For any three pairwise intersecting lines $a,b,c$ there is a fourth line $d$ and three distinct points $A,B,C$ such that the respective incidences $AId$, $AIa$, $BId$, $BIb$ and $CId$, $CIc$ are hold.
\end{enumerate}

Our definition of ordered plane is:

\begin{defi}\label{def:orderedplane} The plane is \emph{ordered} if it satisfies the axioms $\mathrm{I}1-\mathrm{I}4$, $\mathrm{B}1-\mathrm{B}3$ and $\mathrm{B}4^\star$, respectively. The plane is \emph{non-Paschian ordered plane} if it is an ordered plane without $\mathrm{B}4^\star$.
\end{defi}

\begin{remark} 
In the paper \cite{kreuzer} we can find the characterization of the so-called halfordered planes. Among these types one gives a non-trivial example for non-Paschian ordered plane. In fact, the halfordered plane with the property "from three distinct collinear points are exactly two between the others" characterized as an affine plane with order $5$ with the following betweenness property: If $A,B$ two points on a line and the other points are in the form $C=A+\lambda(B-A)$ where $\lambda\in \mathbb{Z}_5$ then $(BAC)$ holds if and only if $\lambda=4$ (a square in $\mathbb{Z}_5$). If one changes $(ABC)$ with $\neg (ABC)$ one gets an ordered plane with $\neg\mathrm{B}4$. But this is a finite model without any nice metric property as e.g. the existence of the angle of lines with a corresponding angular measure. The purpose of this paper to give another model (with interesting geometric properties) such that which does not arise from any usual affine structure.
\end{remark}

Note that in a classical $2$-dimensional absolute plane (holding axioms $\mathrm{I}1-\mathrm{I}3$, $\mathrm{B}1-\mathrm{B}4$) the concept of segment can be defined, as the set of points between the two end-point of the segment, and the following five statements are valid:

\begin{itemize}
\item[T1.] Every segment has at least one point. (Which is according to our definition not an endpoint of the segment.)
\item[T2.] The so-called degenerated case of the Pasch axiom holds, which says that if a line intersects a segment which is one of the three segments determined by three collinear points, then it intersects at least one of the other two segments.
\item[T3.] For four collinear points $A,B,C,D$ for which $(ABC)$ and $(ACD)$ hold, we also have $(ABD)$. (With other words all points of a segment lying on a segment which contains the endpoints of the first one). Consequently each segment and each line has infinitely many points.
\item[T4.] In a line, a point $O$ determines two (closed) half-lines with the common origin $O$ which is the only common point of them.
\item[T5.] A line $l$ determines two (closed) half-planes with the intersection line $l$ which is the common boundary of the two half-planes.
\end{itemize}

Given a line $l$ and a point $O$ on it. The points $A$ and $B$ on $l$ are locate on \emph{opposite sides} of $O$ if $(AOB)$. The points $A$ and $B$ are on the \emph{same side} of $O$ if $(OAB)$ or $(OBA)$. An (open) \emph{half-line with the origin $O$} contains all the points which are on the same side of $O$. [T4] says that the property "same-side" is an equivalence relation with two classes. Analogously we can define the (open) half-plane as a class of the equivalence relation based on the concepts of "opposite sides" and "same side" of $l$. [T5] requires that for a line $l$ the corresponding equivalence relation has exactly two equivalence classes.

The model of our non-Paschian ordered plane is based on the known model of a finite line, which is the so-called \emph{line-model of five-points}. In this model a betweenness relation is defined satisfying the axioms $\mathrm{B}1,\mathrm{B}2,\mathrm{B}3$. (Thus the  model shows that the axioms $\mathrm{B}1,\mathrm{B}2,\mathrm{B}3$ do not imply that a line has infinitely many points (see in \cite{moussong} or \cite{gho}). Using this model we can construct a non-Paschian ordered plane having infinitely many lines with five points and five lines with infinitely many points.

Let $A,B,C,D$ and $E$ be the vertices of a regular pentagon. Clearly, all triangles determined by three vertices of a regular pentagon are isosceles triangles but not equilateral ones. Hence in each triangle  we can chose one of the three vertices (in a natural way) which is in between the other two points: the common vertex of the two legs of the triangle. (So if the legs are $AB$ and $BC$ we define the betweenness such that $B$ is between the points $A$ and $C$.)
The properties of a five-points line are the following:
\begin{itemize}
\item It is easy to see that such line fulfils the axioms  $\mathrm{B}1,\mathrm{B}2$ and $\mathrm{B}3$.
\item Clearly, the closed segments of the line are the point sets with cardinality three. The two endpoints of a segment are the vertices on the base of the corresponding isosceles triangle. We have two types of segments. One of them called \emph{small segment}, this is formed by two consecutive edges (as legs) of the pentagon. The segment is a \emph{large segment}, if its base is an edge of the pentagon.
\item Theorem [T1] is true but [T2] and [T3] are false, respectively. In fact, from $(ABC)$ follows that $B$ is the only point of the segment $AC$, and from $(CAD)$ follows that the only point of the segment $CD$ is $A$. Since from $(ACD)$ follows that the only point of the segment $AD$ is $C$, thus a line (of an embedding plane) distinct from $ABCDE$ can intersect exactly one edge of the degenerated triangle $ACD$. Moreover $(ABC)$ and $(ACE)$ does not implies $(ABE)$ but also $(EAB)$ holds.
\item Consider the true relations $(ABC)$, $(ACE)$ and $(EAB)$. Assume that these relations arise from an ordering of the points $A,B$ and $C$, this ordering is either $A<B<C$ or $A>B>C$. If we assume the first possibility then $A<B<C$ implies $A<C$ and hence $A<E$ holds, too. Then the third betweenness relation implies that $A>B$ contradicting with our first assumption. The second possibility $A<B<C$ leads to the same contradiction, meaning that there is no ordering of the points which can imply the relation of betweenness.
\item The linear axiom of congruence holds for the equivalence relation on this line in which two large (or two small) segments are congruent to each other, respectively. We can also define the union of two segments. Two segments with one common endpoint have the union as the segment corresponding to the free ends of the given ones. The lengths of the segments can be prescribed as the corresponding elements of the finite field $GF(3)$. ($GF(3)$ is the only finite field with three elements which isomorphic to the ring of integers modulo $3$).
\item Observe that the continuity axiom of Archimedes, Cantor and Dedekind hold, respectively. In fact, Cantor's axiom is true because there is no two segments containing each other. Dedekind's axiom is true, because there are no two sets $\mathcal{A}$, $\mathcal{B}$ in the line forming a Dedekind cut. Finally, the axiom of Archimedes follows from the fact that every large segment is a subset of the union of two small ones; and every small segment is a subset of the union of two big ones.
\end{itemize}

\begin{figure}
\centering
\includegraphics{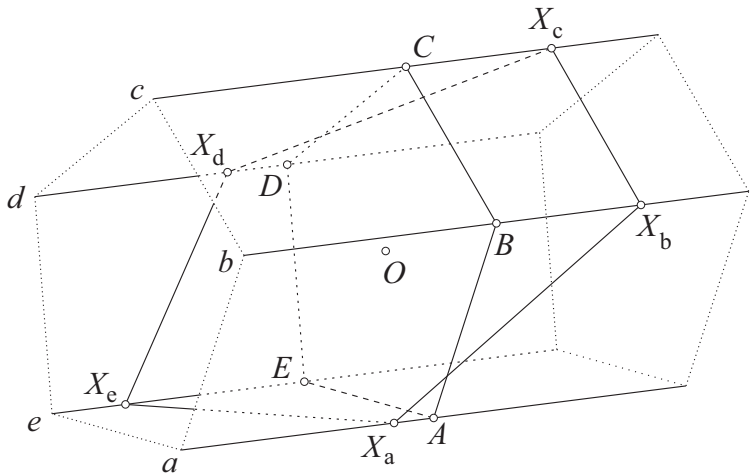}\\
\caption{Two non-intersecting lines}\label{model}
 \end{figure}

Let now $a,b,c,d,e$ be five parallel lines orthogonal to the plane of the points $A,B,C,D,E$, and passing through the points $A,B,C,D,E$, respectively (see Fig.\ref{model}). Denote the center of the regular pentagon by $O$ and consider the 2-planes of the Euclidean space through this point. Let the set of points $\mathcal{P}$ be the collection of the points of the $3$-space which belong to the lines $a,b,c,d,e$, and the set of lines $\mathcal{L}$ is the union of $a,b,c,d,e$ with the five-points sets ${X_a,X_b,X_c,X_d,X_e}$ obtained by the intersection of the above mentioned $2$-planes and the lines $a,b,c,d,e$.
The concept of incidence is the same as the Euclidean incidence. It is clear that $\mathrm{I}1-\mathrm{I}3$ are  fulfilled. Define the property of betweenness on the following way. If three points $X$,$Y$,$Z$ lie on one of the lines $a,b,c,d$ or $e$ then let us denote by $(XYZ)$ the fact that $Y$ lies between the points $X$ and $Z$ with respect to the Euclidean concept of betweenness. In the  case if three points $X_j\in j$, $X_k\in k$ and $X_l\in l$ are such that $\{j,k,l\}\subset \{a,b,c,d,e\}$ - lying on the same plane - we denote by $(X_j,X_k,X_l)$ the fact that $(JKL)$ holds with respect to the concept of betweenness in the five-points model of the points $\{A,B,C,D,E\}$.

\begin{defi}\label{def:angle}
Two planes which intersect the model in the points of two lines $l_1,l_2\in\mathcal{L}$ divide the embedding space into four parts where the opposite parts are congruent to each other, respectively. The common line $m$ of the two planes contains at most one point of the model. In this  case let us denote by $P=l_1\cap l_2\in a$ the common  point of the lines. We can assign to these lines $l_1,l_2$ two \emph{angles}. They contain all points of the model  which belong to the union of the opposite wedge of the space, determined by the given planes, respectively. The legs of these angles are the lines $l_1$ and $l_2$, the common vertex of the angles is $P$. The \emph{angle measures} of the domains are the Euclidean angle measures of dihedral angle. It implies, that the sum the two non-congruent angles is equal to $\pi$.
\end{defi}

\begin{remark} With respect to Def. \ref{def:angle} we have orthogonality in the model. For instance the lines $a$ and $ABCDE$ are orthogonal to each other. We have a natural concept of triangle, three pairwise intersecting lines $l_1$, $l_2$ and $l_3$ (which three have no point in common) determine three points $L_i=l_j\cap l_k$ where $i,j,k\in\{1,2,3\}$ distinct indices. The vertex set of the triangle is $\{L_1,L_2,L_3\}$, the open sides of the triangle are the three segments $\overline{L_iL_j}$, and the triangle domain is the intersection of three angles corresponding to the pairs of lines $l_i,l_j$, respectively (see Fig. \ref{triangle}). The domain of the open triangle depends on the choice of the open angle domains. For each pair of lines we have two possibilities to define it. In Fig. \ref{triangle} $l_1$,$l_2$ and $l_3$ contain the points $X_i$,$Y_i$ and $Z_i$, respectively. The vertices of the triangle are $P=X_a=Y_a$, $Q=Y_b=Z_b$ and $R=X_c=Z_c$, the closed sides are $\mathrm{cl}(\overline{PQ})=\{Y_a,Y_b,Y_d\}$, $\mathrm{cl}(\overline{QR})=\{Z_b,Z_c,Z_e\}$ and $\mathrm{cl}(\overline{PR})=\{X_a,X_c,X_b\}$, respectively. We choose  the angle domain of $l_1,l_2$ with the inner points $\mathrm{int}(QPR\angle):=\overline{X_bY_b}\cup \overline{X_cY_c}\cup \overline{X_dY_d}\cup \overline{X_eY_e}$; for $l_2,l_3$ the angle domain $\mathrm{int}(PQR\angle):=b\setminus {Q}\cup (a\setminus \mathrm{cl} \{\overline{Y_aZ_a}\})\cup (c \setminus \mathrm{cl} \{\overline{Y_cZ_c}\})\cup (d \setminus \mathrm{cl} \{\overline{Y_dZ_d}\})\cup (e \setminus \mathrm{cl} \{\overline{Y_eZ_e}\}) $ and for $l_1,l_3$ let the angle domain be $\mathrm{int}(QRP\angle):=\overline{X_aZ_a}\cup \overline{X_bZ_b}\cup \overline{X_dZ_d}\cup \overline{X_eZ_e}$. The interior of the triangle contains the intersection of these angle domains hence $\mathrm{int}(PQR_\triangle)=\overline{Z_bX_b}=\overline{QX_b}$.
\end{remark}

\begin{figure}
  \centering
    \includegraphics{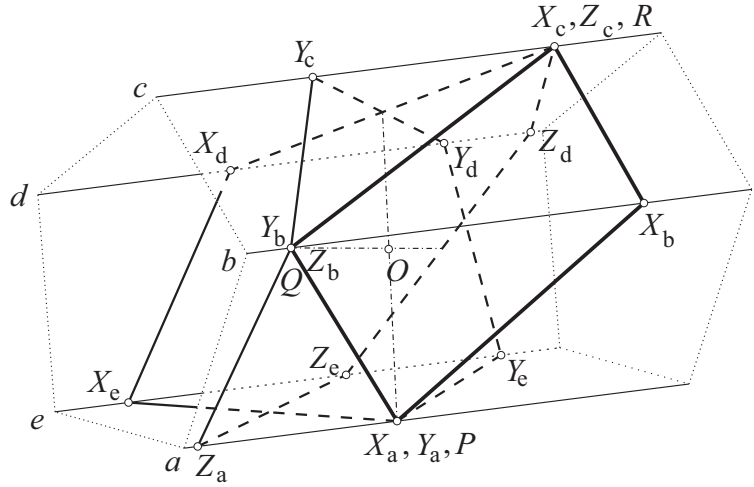}\\
  \caption{Angles and triangles}\label{triangle}
\end{figure}

\begin{theorem}
This construction defines an elementary model of a  non-Paschian ordered plane in the sense of Definition \ref{def:orderedplane}. The three axioms of continuity are true in the corresponding plane. In addition, the sum of the angle measures of the triangles are greater then $\pi$. Moreover there is no universal description of parallelism, there are non-intersecting lines, there is such a pair of points and lines for which the Euclidean axiom of parallels  holds, and there is such a pair of points and lines for which the negation of the Euclidean axiom of parallels is true.
\end{theorem}

\begin{proof}
The axioms of incidence $\mathrm{I}1$, $\mathrm{I}2$ and  $\mathrm{I}3$ are hold, because three non-collinear points of the Euclidean space uniquely determine a plane of the space. To prove $\mathrm{I}4$ observe that the lines $a,b,c,d,e$ intersect all line of five-points, respectively. This implies that three lines
of five-points can be intersected in three distinct points by at least two from the lines $a,b,c,d,e$. In the remaining cases, when among the examined lines we have one or more from the lines $a,b,c,d,e$ we also can guarantee easily common transversalis. In fact, if $p$ and $q$ are two lines of five-points and the third is $a$ then the line $r$ of five-points determined by $P_b=p\cap b$ and $Q_c=q\cap c$ intersects $a$ in a third point $R_a=r\cap a$. If $p$,$a$ and $b$ gives the examined triplet then arbitrary line of five-points through $p_C=p\cap c$ is usable as a transversalis and in the last case when our choice is the three lines $a$,$b$ and $c$ then we can consider any line of five-points to this purpose.

We saw that the axioms B1,B2 and B3 are true on a line of five-points and the remaining lines $a,b,c,d,e$ have Euclidean ordering. In order to prove that $\mathrm{B}4*$ is false we consider the non-collinear triplet $P,Q,R$ and the line $d$. Since the only point of the segment $\overline{PQ}$ is $Y_d\in d$, the line $d$ and its point $Y_d\in \overline{PQ}$ fulfil the assumption of the axiom $\mathrm{B}4^\star$. On the other hand the segments $\overline{QR}=\{Z_e\}\in e$ and $\overline{RP}= \{X_b\}\in b$ are disjoint to $d$ which is a counterexample to $\mathrm{B}4^\star$.

Observe that drawing a unit sphere around the point $O$ we can realize the angle measures of a triangle of the model as the spherical angles of a spherical triangle on the sphere. Hence the sum of the angles of a triangle of the model is equal to the sum of the angles of a spherical triangle which is greater than $\pi$.

Finally, the lines $a$ and $b$ are non-intersecting ones, for the line $a$ through the point $X_b$ we have only one line which does not intersect the line $a$, it is the line $b$. If $l$ is a line of five-points and a point $X$ does not lie on $l$ then there are only finitely many lines through $X$ intersecting $l$, hence the infinitely many other lines through $X$ do not intersect $l$.

\end{proof}

As a closing note we remark that the above-mentioned  construction gives a non-Paschian ordered plane in that case, too, when the axiom of determinateness is used without the axiom of choice to construct the set theory (or to construct any other part of mathematics).

\section*{Acknowledgments}

The author is indebted to the editor for his help to make precise and sharp the axiomatic embeddedness of the model. Many thanks are due to the reviewer his valuable suggestions. He showed me the non-trivial example of Remark 1.

\end{document}